\documentclass[12pt]{amsart}

\usepackage{amsaddr}
\usepackage[utf8]{inputenc}
\usepackage{mathrsfs}

\usepackage{xcolor}
\usepackage[pagebackref,colorlinks,linkcolor=blue,citecolor=orange,urlcolor=blue,hypertexnames=true]{hyperref}
\usepackage[capitalize,nameinlink]{cleveref}

\usepackage{color, bm, amscd, tikz-cd, amssymb}
\usepackage{color, bm, amscd, tikz-cd}

\usepackage{enumitem}
\usepackage{charter}
\usepackage{newtxmath}

\usepackage{euflag}
\newcommand{\cB}{{\mathcal B}}
\newcommand{\cD}{{\mathcal D}}
\newcommand{\cE}{{\mathcal E}}
\newcommand{\cF}{{\mathcal F}}
\newcommand{\cH}{{\mathcal H}}
\newcommand{\cM}{{\mathcal M}}

\newcommand{\la}{\langle}
\newcommand{\ra}{\rangle}

\newcommand{\bfM}{\textit{\textbf{M}}}
\newcommand{\bfT}{\textit{\textbf{T}}}

\newcommand{\bfZ}{\textit{\textbf{Z}}}

\newcommand{\bydef}{\stackrel{\rm def}{=}}

\newtheorem{theorem}{Theorem}[section]
\newtheorem{lemma}[theorem]{Lemma}

\theoremstyle{definition}
\newtheorem{definition}[theorem]{Definition}

\newtheorem{proposition}[theorem]{Proposition}

\theoremstyle{remark}

\numberwithin{equation}{section}

\newtheorem{theoremAinner}{\textbf{Theorem A}}

\newtheorem{theoremBinner}{\textbf{Theorem B}}

\crefname{theoremAinner}{\textbf{Theorem A}}{\textbf{Theorems A}}
\Crefname{theoremAinner}{\textbf{Theorem A}}{\textbf{Theorems A}}

\crefname{theoremBinner}{\textbf{Theorem B}}{\textbf{Theorems B}}
\Crefname{theoremBinner}{\textbf{Theorem B}}{\textbf{Theorems B}}

\crefname{theoremCinner}{\textbf{Theorem C}}{\textbf{Theorems C}}
\Crefname{theoremCinner}{\textbf{Theorem C}}{\textbf{Theorems C}}

\crefname{corollaryCinner}{\textbf{Corollary C}}{\textbf{Corollaries C}}
\Crefname{corollaryCinner}{\textbf{Corollary C}}{\textbf{Corollaries C}}

\thanks{Both the authors were supported by the Slovenian Research Agency program P1-0222 and grant J1-50002. This work was performed
within the project COMPUTE, funded within the QuantERA II Programme, which has
received funding from the European Union's Horizon 2020 research and innovation
programme under Grant Agreement No.~101017733 {\normalsize\protect\euflag}.}

\begin{document}

\title[Taylor Joint Spectrum]{Spectra of $1/k$-Contractions via Characteristic Functions}

\author[Jindal and Kumar]{Abhay Jindal and Poornendu Kumar}

\address{Faculty of Mathematics and Physics, University of Ljubljana, Slovenia\\
\href{mailto:abhay.jindal@fmf.uni-lj.si}{abhay.jindal@fmf.uni-lj.si, poornendu.kumar@fmf.uni-lj.si}}

\bigskip

\dedicatory{
Dedicated to Professor Tirthankar Bhattacharyya}

\subjclass[2020]{47A13, 47A10, 46E22, 47A45}
\keywords{Taylor joint spectra, $1/k$-contractions, inner multipliers, Characteristic functions, Complete Nevanlinna-Pick kernels}

\maketitle

\begin{abstract}
A unitarily invariant complete Nevanlinna--Pick (CNP) kernel $k$ on the Euclidean ball gives rise to a natural class of operator tuples on Hilbert spaces, known as $1/k$-contractions. We establish a lower estimate for the Taylor joint spectrum of $1/k$-contractions with finite defect in terms of the associated characteristic function. Under the additional assumption that $k$ satisfies the Corona property, this lower estimate coincides with an upper estimate due to Clou\^atre--Timko [Adv.~Math., 2023], yielding an exact characterization of the Taylor joint spectrum in terms of the characteristic function. As an application, we determine the Taylor joint spectrum of quotient modules of CNP spaces. A key ingredient in the proof is a Beurling--Lax--Halmos theorem for these spaces, established in terms of characteristic functions. 

\end{abstract}

\maketitle

\section{Introduction}
One of the central themes in multivariable operator theory is the study of the joint spectrum of commuting operator tuples. It has attracted considerable attention and has been investigated from a variety of perspectives in both the classical literature \cite{Curto, Mullar, Taylor, Vas} and more recent works; see, for instance, \cite{ABLS, BRK, BKS3, CT, DEHS, Gleason} and the references therein. In this paper, we study the joint spectrum of a class of commuting operator tuples from the perspective of model theory.

\smallskip

In the single-variable setting, one of the classical frameworks is the Sz.-Nagy--Foiaș theory of contractions, which concerns bounded operators $T$ on a Hilbert space satisfying $\|T\|\le 1$. They introduced an operator-valued bounded analytic function, known as the characteristic function $\theta_T$ of a contraction, and used it to obtain a precise description of the spectrum of a completely non-unitary contraction $T$. More precisely, they \cite{NF} showed that
\begin{align}\label{spectra_char}
\sigma(T) \,=\, \{\lambda \in \mathbb{D} : \theta_T(\lambda)\ \text{is not boundedly invertible}\}\cup \mathcal S_{\theta_T}, 
\end{align}
where $\mathcal S_{\theta_T}$ is the set of boundary points $\lambda \in \mathbb T$ which do not lie on an open arc of $\mathbb T$ across which $\theta_{T}$ extends analytically as a unitary operator-valued function and $\sigma(T)$ is the spectrum of $T$. 

\smallskip

There is another approach through weak-$*$ closed ideals. Given a contraction $T$, let $\operatorname{Ann}(T)$ denote its annihilator in $H^\infty(\mathbb D)$, the Banach algebra of bounded holomorphic functions on $\mathbb D$. If $\operatorname{Ann}(T)$ is nontrivial, then it is a weak-$*$ closed ideal of $H^\infty(\mathbb D)$. As a consequence of Beurling’s theorem, $\operatorname{Ann}(T)$ is singly generated by an inner function $\theta$ on $\mathbb{D}$. Sz.-Nagy and Foiaș \cite{NF} established a fundamental connection between the so-called the support of the inner function $\theta$ and the spectrum of $T$. More precisely, they proved that
\begin{align}\label{spcctra_ann}
\sigma(T)=\operatorname{supp}(\theta).
\end{align}
Recall that a contraction \(T\) on a Hilbert space is naturally associated
with the Szegő kernel $
s(z,w) \,=\, \frac{1}{1-z\overline w}$
on \(\mathbb D\), in the sense that
\[
\left(\frac{1}{s}\right)(T,T^*)=I-T  T^{*} \geq 0
\]
is precisely the contractivity condition. Motivated by this observation, one may try to attach analogous positivity
conditions to more general reproducing kernels. This leads to the notion of a
\(1/k\)-contraction, which has been studied in various forms; see for instance
\cite{Agler82, AEM02, AE, MV}.  CNP kernels form a particularly important class in this
setting. Examples of CNP kernels include the Drury--Arveson kernel and the Dirichlet kernel. CNP kernels arise naturally from the Nevanlinna--Pick interpolation problem \cite{AM02, Q93}
and have been studied extensively from function-theoretic, operator-theoretic,
and operator-algebraic perspectives. When \(k\) is the
Drury--Arveson kernel, the notion of $1/k$-contraction reduces to that of a commuting row contraction. For a unitarily invariant CNP kernel $k$, the class of $1/k$-contractions has been studied extensively in recent years; see \cite{BJ, CH, CT, PRZ} and the references therein.
We discuss in detail in Subsection \ref{subsection_Char} about $1/k$-contraction. In this note, we continue the study of these tuples and their spectral properties. 

\smallskip

In the multivariable setting, there are several notions of joint spectrum for commuting operator tuples $\bfT = (T_{1}, \dots, T_{d})$. Among them, the Taylor joint spectrum, denoted by $\sigma(\bfT)$ and introduced by Taylor \cite{Taylor}, is perhaps the most prominent due to its rich holomorphic functional calculus \cite{Taylor2}; a brief review is provided in Subsection \ref{subsection_Taylor}. More recently, Clou\^atre and Timko \cite{CT} studied multivariable extensions of \eqref{spcctra_ann}. Subsequently, Didas, Eschmeier, Hartz, and  Scherer  \cite{DEHS} obtained an extension of \eqref{spectra_char} for certain commuting row contractions. See also \cite{BT}, where an analogue of \eqref{spectra_char} was established inside the Euclidean unit ball $\mathbb{B}_d$. Subsequently, analogous extensions of \eqref{spcctra_ann} for various classes of commuting tuples of contractions were obtained in \cite{BKS3, CK}. However, a complete description of the Taylor joint spectrum of \(1/k\)-contractions is still not known. In this paper, we focus on the first approach and obtain a description of the Taylor joint spectrum for certain \(1/k\)-contractions, thereby extending \eqref{spectra_char} to this setting.

\smallskip
The characteristic function for \(1/k\)-contractions on \(\mathbb{B}_d\) has recently been developed in \cite{BJ}; we will discuss it briefly in \cref{subsection_Char}. Since the Taylor joint spectrum of a \(1/k\)-contraction is contained in \(\overline{\mathbb{B}}_d\) (see \cite[Lemma 5.3]{CH}), a natural first step is to extend the characteristic function to the boundary.
In the case where $k$ is the Drury--Arveson kernel, the characteristic function admits a holomorphic extension to a neighborhood of the complement of the spectrum; see \cite[Section 4]{DEHS}. A key feature of the Drury--Arveson setting is that $1/k$ is a polynomial, which plays a crucial role in defining the characteristic function on the boundary. 

\smallskip

For a general CNP kernel, however, $1/k$ need not be a polynomial, and this presents a first obstacle to extending the characteristic function to the boundary. Thus, we first extend the characteristic function continuously to the complement of the Taylor joint spectrum of the tuple in the closed unit ball. This is established in Proposition \ref{prop:continuous-extension-characteristic-function}. This extension allows us to relate the Taylor joint spectrum of such a tuple to its associated characteristic function. We first show that, when the defect operators have finite rank, the Taylor joint spectrum contains all the points at which the characteristic function fails to be surjective or does not extend; see \cref{theorem_tylor_lower}. Next, we prove that if the CNP kernel $k$ has the {\em corona property}, then equality holds in \cref{theorem_tylor_lower}. These results are summarized in Theorem A below. Before proceeding, we introduce the following notation: Given an inner multiplier $
\theta : \mathbb{B}_d \to \mathcal{B}(\mathcal{F},\mathcal{E})$ (in the sense of McCullough--Trent \cite{MT}),
where $\mathcal{E}$ and $\mathcal{F}$ are Hilbert spaces, we define
\[
E_{\theta} = \Big\{\, \bm\lambda \in \overline{\mathbb{B}}_d \ :\  
\lim_{\substack{\bm z \to \bm\lambda\\ \bm z \in \mathbb{B}_d}} \theta(\bm z)
\ \text{exists and is not surjective} \Big\}, \text{ and } 
\]
\[
F_{\theta} = \Big\{\, \bm\lambda \in \overline{\mathbb{B}}_d \ :\  
\lim_{\substack{\bm z \to \bm\lambda\\ \bm z \in \mathbb{B}_d}} \theta(\bm z)
\ \text{does not exist} \Big\}.
\]

\begin{theoremAinner}\label{thmA}
Let $k$ be a regular unitarily invariant CNP kernel on $\mathbb{B}_d$, and let ${\bfT}=(T_1,\dots,T_d)$ be a pure $1/k$-contraction. Let $
\theta_{{\bfT}} : \mathbb{B}_d \to \mathcal{B}\big(\mathcal{D}_{\tilde{\bfT}}, \overline{\rm Ran}{\Delta}_{\bfT}\big)$
be the characteristic function of ${\bfT}$. Then the following hold:
\begin{enumerate}
    \item The characteristic function $\theta_{\bfT}$ extends continuously to $
    \overline{\mathbb{B}}_d \setminus \sigma({\bfT}).$

    \item If $\dim(\mathrm{Ran}\,\Delta_{\bfT}) < \infty$, then
    \[
    \sigma({\bfT}) \,\supset\, E_{\theta_{\bfT}} \cup F_{\theta_{\bfT}}.
    \]

    \item Moreover, if $\dim(\mathrm{Ran}\,\Delta_{\bfT}) < \infty$ and $k$ has the corona property, then
    \[
    \sigma({\bfT}) \,=\, E_{\theta_{\bfT}} \cup F_{\theta_{\bfT}}.
    \]
\end{enumerate}
\end{theoremAinner}
The inclusion ``$\subset$'' in~(3) follows from a result of Clou\^atre--Timko \cite{CT} together with Lemma~\ref{lem:determinant-ideal-CNP}, we record this consequence in Lemma~\ref{lemma_upper}. In the special case where $k$ is the Drury--Arveson kernel, the above theorem was proved in \cite{DEHS}. Furthermore, the corona property is known to hold for a number of other CNP kernels, including the Dirichlet space \cite{Luo} and certain Besov--Sobolev spaces \cite{CSW}.

\smallskip

The next natural question is to determine the Taylor joint spectrum of quotient modules over vector-valued CNP spaces. 
% To address this problem, one first needs to understand the \(\bfM_{\bm z}\)-invariant subspaces.
A notable result of McCullough and Trent \cite{MT} shows that every \(M_z\)-invariant subspace arises as the range of an inner multiplier. On the other hand, the characteristic function associated with a pure \(1/k\)-contraction is also inner. This naturally leads to the question of understanding the relationship between inner multipliers and characteristic functions.  

To address this question, we first show that every inner multiplier can be realized as the characteristic function of a suitable pure \(1/k\)-contraction; see \cref{theorem_main2}. The proof hinges on understanding the structure of the smallest reducing subspace that contains \(\mathcal{M}^\perp\), with \(\mathcal{M}\) an \(\bfM_{\bm z}\)-invariant subspace. In the classical setting, as well as for the Drury--Arveson space (see for example, \cite[Theorem 5.7]{Arveson}), this follows from the fact that $
\overline{\operatorname{span}}
\left\{
\bfM_{\bm z}^\alpha (\bfM_{\bm z}^*)^\beta
:\;
\alpha,\beta \in \mathbb Z_+^d
\right\}$
forms a \(C^*\)-algebra, see \cite[Theorem 5.7]{Arveson}. Consequently, the smallest reducing subspace containing \(\mathcal{M}^\perp\) is given by $
\overline{\operatorname{span}}
\left\{
\bfM_{\bm z}^\alpha(\mathcal{M}^\perp)
:\;
\alpha\in\mathbb Z_+^d
\right\}.$
However, for general CNP spaces, it is not known whether the above operator space forms a \(C^*\)-algebra. Thus, a different approach is required. Therefore, we adopt a different approach and establish the result in Proposition~\ref{lem:smallest reducing subspace}. Combining this with \cref{thmA}, we determine the Taylor joint spectra of quotient modules over CNP spaces. This yields our second main result.

\begin{theoremBinner}\label{thmB}
Let $k$ be a regular unitary invariant CNP kernel on $\mathbb{B}_{d}.$ Let $\cD$ be a finite-dimensional Hilbert space, $\cM \subseteq H_{k} \otimes \cD$  be a closed subspace invariant under $M_{z_{i}} \otimes I_{\mathcal D}$  on $H_{k} \otimes \cD$ for all $i =1, \ldots, d,$ and $\theta \in {\rm Mult}(H_{k} \otimes \cE, H_{k} \otimes \cD)$ an inner multiplier with $\cM = \theta (H_{k} \otimes \cE).$ Then
\[
\sigma (\bfM_{\bm z} \otimes I_{\cD}, (H_{k} \otimes \cD)/\cM ) \supset E_{\theta} \,\cup\, F_{\theta}.\]
\end{theoremBinner}
 
Our work is motivated by the recent works of Clou\^atre--Timko \cite{CT} and Didas--Eschmeier--Hartz--Scherer \cite{DEHS}. We extend these results to the setting of CNP spaces, made possible by the recent development of characteristic functions in the CNP spaces in \cite{BJ}.

\section{Preliminaries}
	\subsection{Taylor Joint Spectrum:}\label{subsection_Taylor}

We briefly recall the Taylor joint spectrum introduced by Taylor \cite{Taylor}. Let \(E\) denote the exterior algebra generated by \(d\) indeterminates \(e_1,\ldots,e_d\) together with the identity element \(e_0 \equiv 1\). Thus, \(E\) consists of arbitrary linear combinations of products of the \(e_i\)'s subject to the anti-commutation relations
$
e_i e_j = - e_j e_i, \qquad 1 \leq i,j \leq d.$
For each \(i\), define the creation operator \(C_i : E \to E\) by $C_i \xi = e_i \xi$ with $\xi \in E.$
Declaring the set $
\{e_{i_1}\cdots e_{i_l} : 1 \leq i_1 < \cdots < i_l \leq d\}$ to be an orthonormal basis endows \(E\) with a Hilbert space structure. Moreover, \(E\) admits the orthogonal decomposition $
E = \bigoplus_{l=0}^{d} E_l,$ where \(E_l\) consists of all linear combinations of products involving exactly \(l\) elements from \(\{e_1,\ldots,e_d\}\).

Let $
\mathbf{T}=(T_1,\ldots,T_d)$ be a commuting \(d\)-tuple of bounded operators on a normed space \(X\). Consider the space $
E(X):=X\otimes E,$ and define the operator $
\Lambda_{\mathbf{T}}:X\otimes E\to X\otimes E$ by
\[
\Lambda_{\mathbf{T}}
=
\sum_{i=1}^{d} T_i\otimes C_i.
\]
The decomposition $
E=\bigoplus_{l=0}^{d} E_l$ induces the decomposition $
E(X)=\bigoplus_{l=0}^{d} E_l(X)$ where $E_l(X):=X\otimes E_l,$
and gives rise to the Koszul complex associated with \(\mathbf{T}\):
\[
K(\mathbf{T},X):
0
\longrightarrow
E_0(X)
\xrightarrow{\Lambda_0}
E_1(X)
\xrightarrow{\Lambda_1}
\cdots
\xrightarrow{\Lambda_{d-1}}
E_d(X)
\longrightarrow
0,
\]
where \(\Lambda_k\) denotes the restriction of \(\Lambda_{\mathbf{T}}\) to \(E_l(X)\). The Taylor joint spectrum of \(\mathbf{T}\) is defined by
\[
\sigma(\mathbf{T})
=
\left\{
\lambda\in\mathbb{C}^d:
K(\mathbf{T}-\lambda,X)
\text{ is not exact}
\right\}.
\]
It follows from \cite[Proposition 25.3 and Theorem 25.4]{Mullar} that the Taylor joint spectrum $\sigma(\bfT)$ is a compact subset of $\mathbb{C}^d$. We shall also require several basic properties of the Taylor joint spectrum. Denote by $\sigma_p(\bfT)$ the (possibly empty) joint point spectrum of $\bfT$, that is, the set of all joint eigenvalues. It follows from \cite[Remark 25.2]{Mullar} that
\begin{equation}\label{equation_Tylor2}
\sigma_p(\bfT)\subset \sigma(\bfT).
\end{equation}
If $\bfT^*=(T_1^*,\ldots,T_d^*)$, then the discussion in \cite[Section 3]{Richter-Sundberg} shows that
\begin{equation}\label{equation_Tylor1}
\sigma(\bfT^*)=\{\overline{\bm z} : \bm z\in \sigma(\bfT)\}.
\end{equation}
Finally, in the finite-dimensional setting, the Taylor joint spectrum coincides with the joint point spectrum; that is, $
\sigma(\bfT)=\sigma_p(\bfT).$

    \subsection{Operators Associated with CNP Kernels and the Characteristic Function}\label{subsection_Char}
	A reproducing kernel $k$ on $\mathbb B_{d}$
	is said to be a complete  Nevanlinna-Pick (CNP) kernel if for any natural numbers $m,n,$ any $N$ points $\bm \lambda_{1}, \dots, \bm \lambda_{N}$ in $\mathbb{B}_{d}$ and any $m \times n$ matrices $W_{1}, \dots, W_{N},$ the condition that the $N \times N$ block matrix
	$$\Big( (I - W_{i}W_{j}^{*}) k(\bm \lambda_{i}, \bm \lambda_{j}) \Big)_{i,j=1}^{N}$$
	is positive semidefinite implies that there is a holomorphic function $\varphi: \mathbb{B}_{d} \to \mathbb{M}_{m \times n} (\mathbb{C})$ mapping $H_k \otimes \mathbb{C}^{n}$ into $H_k \otimes \mathbb{C}^{m}$ by multiplication and interpolating $\bm \lambda_{i}$ to $W_{i}, i=1, \dots, N.$ The relation to the classical Nevanlinna-Pick interpolation problem is not hard to imagine, hence the name.
	
	\medskip
	
Consider a reproducing kernel 
	\begin{equation}\label{unitarily invariant}
		k: \mathbb{B}_{d} \times \mathbb{B}_{d} \to \mathbb{C}; \quad k(\bm z, \bm w) = \sum\limits_{n=0}^{\infty} a_{n} \la \bm z , \bm w \ra^{n}, \quad (\bm z ,\bm w \in \mathbb{B}_{d})
	\end{equation}
	with $a_{0} = 1$ and $a_{n} > 0$ for all $n \geq 1.$  It is clearly {\em unitarily invariant}, i.e.,
	\[
    k(U \bm z, U \bm w) = k(\bm z, \bm w) \quad \text{for all $d\times d$ unitary matrices $U$}.
    \]
    A unitarily invariant kernel of the form \eqref{unitarily invariant} is called {\em regular} if the positive coefficients $a_{n}$ satisfies
    \[
    \lim\limits_{n \rightarrow \infty} \frac{a_{n}}{a_{n+1}} = 1.
    \]
    A reproducing kernel $k : \mathbb B_{d} \times \mathbb B_{d} \to \mathbb C$ is called {\em irreducible} if $k (\bm z, \bm w) \neq 0$ for all $ \bm z, \bm w \in \mathbb{B}_{d}$ and $k_{\bm w}$ and $k_{\bm \nu}$ are linearly independent if $\bm \nu \neq \bm w$ where $k_{\bm w}(\bm z) = k(\bm z, \bm w).$ 

    \smallskip

	There is a characterization of unitarily invariant CNP kernels which we shall greatly use, viz., a unitarily invariant kernel $k$ is irreducible and CNP  if and only if there is a sequence of non-negative real numbers $\{b_{n}\}_{n=1}^{\infty}$ such that
	\begin{equation} \label{bn}
		1 - \frac{1}{k(\bm z, \bm w)} \,=\, \sum\limits_{n=1}^{\infty} b_{n} \la \bm z, \bm w \ra^{n}, \qquad (\bm z, \bm w \in \mathbb{B}_{d}).
	\end{equation}
	\begin{definition}\label{def}
	A regular unitarily invariant kernel $k$ is called a {\em unitarily invariant CNP kernel} if  it is a irreducible CNP kernel, or equivalently, there exists a sequence of non-negative real numbers $\{b_{n}\}_{n=1}^{\infty}$ satisfying \eqref{bn}. The corresponding reproducing kernel Hilbert space is denoted by $H_k.$ 
	\end{definition}

Denote by $\mathbb{Z}_{+}$ the set of all non-negative integers. Let $\alpha= (\alpha_{1}, \ldots ,\alpha_{d})\in\mathbb{Z}^{d}_{+}$ be a {\em multi-index}. Let $\bm{z}\in \mathbb{C}^{d}$. We need the following notations.
$$ |\alpha|=\alpha_{1}+ \cdots +\alpha_{d}, \;\; \alpha !=\alpha_{1}! \ldots \alpha_{d}!, \;\; \binom{|\alpha|}{\alpha} = \frac{|\alpha|!}{\alpha_{1}!\dots\alpha_{d}!} \text{ and } \bm z^{\alpha} = z_{1}^{\alpha_{1}}  \ldots z_{d}^{\alpha_{d}}.$$
To simplify notations, we define the coefficients $a_{\alpha}$ and $b_{\alpha}$ as follows:
\begin{equation*}
	a_{\alpha} = \begin{cases}  a_{|\alpha|} \binom{|\alpha|}{\alpha}, & \alpha\in\mathbb{Z}^{d}_{+} \\ 0, & \alpha\in\mathbb{Z}^{d} \backslash \mathbb{Z}^{d}_{+} \end{cases},\quad  \text{and} \quad
	b_{\alpha}= b_{|\alpha|} \binom{|\alpha|}{\alpha},\quad  \alpha\in\mathbb{Z}^{d}_{+} \backslash \{0\}.
\end{equation*}
We shall need, for each multi-index $\alpha\in\mathbb{Z}^{d}_{+} \backslash \{0\},$  the polynomial
\begin{equation} \label{psi} 
	\psi_{\alpha} : \mathbb{B}_{d} \to \mathbb{C}; \quad \bm z \mapsto (b_{\alpha})^{1/2} \bm{z}^{\alpha}.\end{equation}

\medskip

We now turn to $d-$tuples of bounded operators $ \bfT = (T_{1}, \dots, T_{d})$. Any tuple of bounded operators in this note always consists of commuting operators. Set 
$$\bfT^{\alpha} =  T_{1}^{\alpha_{1}} \dots T_{d}^{\alpha_{d}}, \quad \alpha \in \mathbb{Z}^{d}_{+}.$$
The following definition inspired by the expression in \eqref{bn} is introduced in \cite{CH} and plays a vital role in our analysis. 

\begin{definition}
	If the series $\sum\limits_{\alpha\in\mathbb{Z}^{d}_{+} \backslash \{0\}} b_{\alpha} \bfT^{\alpha}(\bfT^\alpha)^{*}$ converges in the strong operator topology to a contraction, then the $d-$tuple $\bfT$ is referred to as a $1/k$-contraction.
	In this case, we denote the unique positive square root of the positive operator
	$$ I-\sum\limits_{\alpha\in\mathbb{Z}^{d}_{+} \backslash \{0\}} b_{\alpha} \bfT^{\alpha}(\bfT^{\alpha})^{*}$$
	by $\Delta_{\bfT}.$  We shall call $\Delta_{\bfT}$ the  defect operator.
\end{definition}

If $k$ is the Drury-Arveson kernel $1/(1 - \langle \bm z, \bm w \rangle)$, then a $1/k$-contraction is same as row contraction and the defect operator $\Delta_{\bfT}$ is $(I - T_1T_1^* - \cdots - T_dT_d^*)^{1/2}$.

\smallskip

 If $\bfT = (T_{1}, \dots, T_{d})$ is a $1/k$-contraction then the series
 \begin{equation} \label{eq:pure}
 \sum\limits_{\alpha \in \mathbb Z^{d}_{+}} a_{\alpha} \bfT^{\alpha} \Delta_{\bfT}^{2} (\bfT^{\alpha})^{*}
 \end{equation}
 converges strongly to a contraction; see \cite[Lemma 4.1]{BJ}.
 
 \begin{definition}
A $1/k$-contraction $\bfT$ is called {\em pure} if the series in \eqref{eq:pure} converges to the identity operator. It is easy to check that the compression of $\bfM_{\bm z} = (M_{z_{1}}, \dots, M_{z_{d}})$ onto a co-invariant subspace is a  pure contraction.  
 \end{definition}

\medskip

One of the principal tools we shall use is the {\em characteristic function} developed in \cite{BJ}. We need to recall the construction. Let
\[
\tilde{\cH} \bydef \oplus_{\alpha\in\mathbb{Z}^{d}_{+} \backslash \{0\}} \cH,
\]
the infinite direct sum of the Hilbert space $\cH.$  Recall the $\psi_\alpha$ from \eqref{psi} and define the infinite operator tuple
\begin{equation*}\label{def Z}
	\bfZ = (\psi_{\alpha}(z) I_{\cH})_{\alpha\in\mathbb{Z}^{d}_{+} \backslash \{0\}}.
\end{equation*}
The same notation $\bfZ$ also serves for the operator from $\tilde{\cH}$ to $\cH$ mapping $(h_{\alpha})_{\alpha\in\mathbb{Z}^{d}_{+} \backslash \{0\}}$ to $\sum\limits_{\alpha\in\mathbb{Z}^{d}_{+} \backslash \{0\}} (b_{\alpha})^{1/2} \bm{z}^{\alpha} h_{\alpha}.$
Similarly,  $\tilde{\bfT}$ stands for the infinite operator tuple
\begin{equation*}\label{def Ttilde}
	\tilde{\bfT}  =  (\psi_{\alpha}(\bfT))_{\alpha\in\mathbb{Z}^{d}_{+} \backslash \{0\}}
\end{equation*}
as well as the operator from $\tilde{\cH}$ to $\cH$ which maps $(h_{\alpha})_{\alpha\in\mathbb{Z}^{d}_{+} \backslash \{0\}}$ to $\sum\limits_{\alpha\in\mathbb{Z}^{d}_{+} \backslash \{0\}} (b_{\alpha})^{1/2} \bfT^{\alpha} h_{\alpha}.$ Note that
$$\|\bfZ\|^{2} = \sum\limits_{\alpha\in\mathbb{Z}^{d}_{+} \backslash \{0\}} b_{\alpha} | \bm{z}^{\alpha}|^{2} = 1 - \frac{1}{s(\bm{z}, \bm{z})} < 1.$$
Moreover, $\tilde{\bfT}$ is a contraction if and only if $\bfT$ is a $1/k$-contraction.  So, $I_{\cH}-\bfZ \tilde{\bfT}^{*}$ is invertible.

\medskip

A straightforward computation shows that $\Delta_{\bfT}^{2} = I_{\cH} - \tilde{\bfT} \tilde{\bfT}^{*}.$ Let $D_{\tilde{\bfT}}$ be the unique positive square root of the positive operator $I_{\tilde{H}} - \tilde{\bfT}^{*} \tilde{\bfT},$ and let $\cD_{\tilde{\bfT}} = \overline{\rm Ran} D_{\tilde{\bfT}}.$ By equation (I.3.4) of \cite{NF} we obtain the identity
\begin{equation}\label{defect}
	\tilde{\bfT}D_{\tilde{\bfT}}= \Delta_{\bfT} \tilde{\bfT}.
\end{equation}

\begin{definition}\label{def_char}
	The characteristic function of a $1/k$-contraction $\bfT = (T_{1}, \dots, T_{d})$ is the analytic operator valued function $\theta_{\bfT}:\mathbb{B}_{d}\to\cB(\cD_{\tilde{\bfT}}, \overline{\rm Ran} \Delta_{\bfT})$  given by
	$$	\theta_{\bfT}(\bm{z})=(-\tilde{\bfT}+ \Delta_{\bfT}(I_{\cH}-\bfZ \tilde{\bfT}^{*})^{-1}\bfZ D_{\tilde{\bfT}})|_{\cD_{\tilde{\bfT}}}. $$
\end{definition}
The characteristic function $\theta_{\bfT}$ takes values in $\cB(\cD_{\tilde{\bfT}}, \overline{\rm Ran} \Delta_{\bfT})$ by virtue of \eqref{defect}. For Hilbert spaces $\cE$ and $\mathcal F$, let ${\rm Mult}(H_k \otimes \cE, H_k \otimes \mathcal F )$ denote the {\em multiplier space}, i.e., the vector space of all $\mathcal B (\cE, \mathcal F)$ valued functions $\varphi$ on $\mathbb B_d$ such that the multiplication operator $M_\varphi$ is in $\mathcal B (H_k \otimes \cE, H_k \otimes \mathcal F )$. The characteristic function  $\theta_{\bfT}$ is  in ${\rm Mult}(H_k \otimes \cD_{\tilde{\bfT}}, H_k \otimes \overline{\rm Ran} \Delta_{\bfT} )$ and is contractive because of  \cite[Theorem 4.11]{BJ}. 

\smallskip

The characteristic function for row contractions was introduced in \cite{BES}, while Popescu developed the corresponding theory for noncommuting row contractions in \cite{Popescu}.

\section{Preparatory Results for the Main Theorems}
The results proved in this section constitute the main tools used in the proof of the main theorems. For clarity, the discussion is divided into subsections. Some of these results are of interest in their own right.
\subsection{Extension of the characteristic function:}
The main purpose of this subsection is to study the extension of the characteristic function  for a given $1/k$-contraction to the boundary. We begin with the following elementary lemma.

\begin{lemma}\label{lem:sum-bn-leq-one}
	Let
	\[
	k(\bm z, \bm w)  \,=\, \sum_{n=0}^{\infty} a_n \langle \bm z, \bm w\rangle^n \,=\, 1 - \frac{1}{\sum_{n=1}^{\infty} b_n \langle \bm z, \bm w\rangle^n}
	\]
	be a unitarily invariant CNP kernel on $\mathbb B_d$.
	Then
	\[
	\sum_{n=1}^{\infty} b_n\leq 1.
	\]
\end{lemma}

\begin{proof}
	Fix $\bm \zeta \in\partial\mathbb B_d$. For $0\leq r<1$,
	\[
	\sum_{n=1}^{\infty}b_n r^{2n}
	\,=\,
	1-\frac{1}{k(r \bm \zeta,r \bm \zeta)}.
	\]
	Since $k(r \bm \zeta, r \bm \zeta)\geq a_0=1$, we have
	\[
	0 \,\leq\, 1-\frac{1}{k(r\bm \zeta, r \bm \zeta)} \,\leq\, 1.
	\]
	Thus $
	\sum_{n=1}^{\infty}b_n r^{2n}\leq 1 $
	for every $0\leq r<1$. Letting $r\uparrow 1$ and using $b_n\geq 0$, we obtain
	\[
	\sum_{n=1}^{\infty}b_n\leq 1.
	\]
\end{proof}
Let $\bfT = (T_{1}, \dots, T_{d})$ be a $1/k$-contraction. As a consequence of Lemma \ref{lem:sum-bn-leq-one}, we get that the series 
\[
\sum\limits_{\alpha \in \mathbb{Z}^{d}_{+} \setminus \{0\}} b_{\alpha} \bm{z}^{\alpha} (\bfT^{\alpha})^{*}
\]
converges in norm for all $\bm z \in \overline{\mathbb{B}}_d.$ The following proposition shows that the characteristic function admits a continuous extension to 
$\overline{\mathbb{B}}_d \setminus \sigma(\mathbf{T})$ for a given $1/k$-contraction.

\begin{proposition}
\label{prop:continuous-extension-characteristic-function}
Let $k$ be a unitarily invariant CNP kernel on $\mathbb B_d$. Let $\bfT=(T_1,\ldots,T_d)$ be a $1/k$-contraction on a Hilbert space $\mathcal H.$ Then the characteristic function $\theta_T$ extends continuously to
$ \overline{\mathbb B}_{d}\setminus \sigma(\bfT).$
\end{proposition}

\begin{proof}
	First we prove that
	$ I_{\mathcal H}-\bfZ\tilde \bfT^*$
	is invertible whenever $ \bm z\in\overline{\mathbb B}_{d}\setminus\sigma(\bfT).$
	Fix such a point $\bm z$, and set 
	\begin{equation} \label{eq:norm convergent series}
		A_{\bm z} 
		\,:=\, 
		\bfZ \tilde \bfT^*
		\,=\,
		\sum\limits_{\alpha \in \mathbb{Z}^{d}_{+} \setminus \{0\}} b_{\alpha} \bm{z}^{\alpha} (\bfT^{\alpha})^{*}.
	\end{equation}
	Let $\mathcal A$ be the unital commutative Banach algebra generated by
	$T_1^*,\ldots,T_d^*$. Then $A_{\bm z}\in\mathcal A$ since the series in \eqref{eq:norm convergent series} converges in norm.  We claim that $I_{\mathcal H}-A_{\bm z}$ is invertible. It is enough to show that its Gelfand transform is nowhere zero on the maximal ideal space of $\mathcal A$.
	
	\smallskip
	
	Let $\varphi$ be a character of $\mathcal A$. Set
	\[
	\mu_j \,:=\, \varphi(T_j^*),
	\qquad j=1,\ldots,d.
	\]
	Then $\bm \mu=(\mu_1,\ldots,\mu_d)\in\sigma(T_{1}^{*}, \dots, T_{d}^{*})$. Equivalently, $ \bm \lambda:=\overline{\bm \mu}\in\sigma(\bfT).$
	Since $\sigma(\bfT)\subseteq \overline{\mathbb B_d}$, we have
	$ \bm \lambda\in\overline{\mathbb B}_{d}.$ Now
	\[
	\varphi(A_{\bm z})
	\,=\,
	\sum_{\alpha\neq 0}b_\alpha \bm z^\alpha \bm \mu^\alpha
	\,=\,
	\sum_{n=1}^{\infty}b_n\langle \bm z, \bm \lambda\rangle^n.
	\]
	Since $\bm z \neq \bm \lambda,$ the right hand side can not be $1.$  Thus, for every character $\varphi$ of $\mathcal A$,
	\[
	\varphi(I_{\mathcal H}-A_{\bm z})
	\,=\,
	1-\varphi(A_{\bm z})
	\,\neq\,
	0.
	\]
	Hence $I_{\mathcal H}-A_{\bm z}$ is invertible in $\mathcal A$, and therefore invertible in $\mathcal B(\mathcal H)$.
	
	\smallskip
	
	Now define, for $\bm z\in\overline{\mathbb B}_{d}\setminus\sigma(\bfT)$,
	\[
	\widehat{\theta}_\bfT(\bm z)
	=
	\left(
	-\tilde \bfT
	+
	\Delta_\bfT
	\left(I_{\mathcal H}- \bfZ \tilde \bfT^*\right)^{-1}
	\bfZ D_{\tilde \bfT}
	\right)\bigg|_{\mathcal D_{\tilde \bfT}}.
	\]
	This is well-defined by the preceding paragraph. It remains to prove continuity. Since $ \bm z\mapsto \bfZ $
	is norm-continuous on $\overline{\mathbb B}_{d}$, the map $ \bm z\mapsto \bfZ \tilde \bfT^* $
	is also norm-continuous on $\overline{\mathbb B}_{d}$. Inversion is norm-continuous on the open set of invertible operators. Therefore
	\[
	\bm z\mapsto
	\left(I_{\mathcal H}- \bfZ \tilde \bfT^*\right)^{-1}
	\]
	is norm-continuous on $ \overline{\mathbb B}_{d}\setminus\sigma(\bfT).$
	The remaining factors $\bfT$, $\Delta_\bfT$, and $D_{\tilde \bfT}$ are fixed bounded operators. Hence
	\[
	\bm z\mapsto \widehat{\theta}_\bfT(\bm z)
	\]
	is norm-continuous on $ \overline{\mathbb B}_{d}\setminus\sigma(\bfT).$
	
	\smallskip
	
	Finally, for $\bm z\in\mathbb B_d$, the above formula agrees with the original characteristic function $\theta_{\bfT} (\bm z)$. Hence $\widehat\theta_{\bfT}$ is a continuous extension of $\theta_{\bfT}$ to $
	\overline{\mathbb B}_{d}\setminus\sigma(\bfT).$
\end{proof}

\subsection{Smallest Reducing Subspace Containing a Co-Invariant Subspace}  We start with the following lemmas.
\begin{lemma}\label{lemma_compact}
	Let $k$ be a unitarily invariant kernel on $\mathbb{B}_{d}$ defined in \eqref{unitarily invariant} satisfying
	\begin{enumerate}
		\item $ \sup\limits_{n} \frac{a_{n}}{a_{n+1}} \,<\, \infty,$ and 
		\item $\lim\limits_{n \rightarrow \infty} \left( \frac{a_{n}}{a_{n+1}} - \frac{a_{n-1}}{a_{n}}  \right) \,=\,  0, $
	\end{enumerate}
	Then for every pair of multi-indices
	\(\alpha,\beta\in \mathbb Z^d_{+}\), the commutator 
	\[
	\left[\bfM_{\bm z}^\alpha,  (\bfM_{\bm z}^{\beta})^{*}\right]
	\,:=\,
	\bfM_{\bm z}^\alpha (\bfM_{\bm z}^{\beta})^{*} - (\bfM_{\bm z}^{\beta})^{*} \bfM_{\bm z}^\alpha
	\]
	is compact on $H_k$.
\end{lemma}
\begin{proof}
	By \cite[Corollary 4.4]{Guo}, the basic commutators $M_{z_{i}} M_{z_{j}}^{*} - M_{z_{j}}^{*} M_{z_{i}}$
	are compact for all $i,j = 1,\dots,d$. Since the space of compact operators forms a two-sided ideal in $B(H_k)$, it follows from well known commutator identities of operators 
	\[
	[XY,Z] \,=\, X[Y,Z] + [X,Z]Y \quad \text{ and } \quad 
	[X,YZ]=[X,Y]Z+Y[X,Z] 
	\]
	that the operators $ \bfM_{\bm z}^\alpha (\bfM_{\bm z}^{\beta})^{*} - (\bfM_{\bm z}^{\beta})^{*} \bfM_{\bm z}^\alpha$ are compact. 
\end{proof}

\begin{lemma}\label{lemma_CI}
	Let $\mathcal E$ be a Hilbert space. Let $\mathcal K$ denote the set of compact operators on $H_{k}$ and let $\cM$ be an $\bfM_{\bm z}$-invariant closed subspace of $H_{k} \otimes \mathcal E$. Then 
	\[
	\mathcal{K} \left( \mathcal{M}^\perp\right) 
	\,:=\,
	\{(X \otimes I_{\mathcal E}) (\mathcal M^{\perp})\ :\ X \in \mathcal K \} 
	\,\subseteq\,
	\overline{\operatorname{span}}\left\{\bfM_{\bm z}^\alpha(\mathcal{M}^\perp) \ :\  \alpha\in\mathbb{Z}_{+}^d\right\}.
	\]
\end{lemma}
\begin{proof}
	Let us denote the closed subspace $
	\overline{\operatorname{span}}
	\left\{
	\bfM_{\bm z}^\alpha(\mathcal M^\perp)
	\ :\ \alpha\in\mathbb Z_+^d
	\right\}$ by $\mathcal U$. It is enough to prove that 
	\[
	(X\otimes I_{\mathcal E}) (\mathcal M^{\perp}) \,\subseteq\, \mathcal U
	\]
	for rank-one $X$, since they span finite-rank operators which are dense in $\mathcal{K}$ and $\mathcal U$ is a closed subspace. A rank one operator on $H_k$ is of the form 
	\[
	f\otimes g^* \,:\, \eta \,\mapsto\, \langle \eta, f\rangle\, g  \qquad (\eta \in H_k)
	\]
	for some $f$ and $g$ in $H_k.$ Since polynomials are dense in $H_k,$ we can assume that $f$ and $g$ are polynomials. In this case, the rank-one operator $f \otimes g^{*}$ can be written as
	$$
	f \otimes g^{*} = \bfM_f \, E_0\, \bfM_g^{*},
	$$
	where $E_0$ denotes the orthogonal projection of $H_k$ onto the one-dimensional subspace of constant functions.  Since $\mathcal{M}$ is $\bfM_{\bm z}$-invariant,
	$\mathcal{M}^\perp$ is $\bfM_{\bm z}^*$-invariant. As $g$ is a polynomial, we therefore obtain
	\[
	\bfM_g^*  (\mathcal{M}^\perp) 
	\,\subseteq\, 
	\mathcal{M}^\perp .
	\]
	Thus, it suffices to show that $(E_0 \otimes I_{\mathcal E}) \left(\mathcal{M}^\perp\right)\subset\mathcal U$. 
	
	\smallskip
	
	Let $h \in \mathcal{M}^\perp$. We want to show that $(E_0 \otimes I_{\mathcal E})  h \in \mathcal{U}$. Recall that the series $ \sum_{\alpha \in \mathbb{Z}^d_{+} \setminus \{0\}} b_{\alpha}\, \bfM_z^{\alpha} (\bfM_z^{\alpha})^{*} $
	converges to $I_{H_k \otimes \mathcal E} - E_{0} \otimes I_{\mathcal E}$ in the strong operator topology. 
	Thus for $\xi \in\mathcal{U}^\perp$ we obtain
	\begin{align*}
		\big\langle (E_{0} \otimes I_{\mathcal E}) h, \xi \big\rangle
		& \,=\, \langle h, \xi \rangle - \Big\langle\sum_{\alpha \in \mathbb{Z}^d_{+} \setminus \{0\}} b_{\alpha}\, \bfM_z^{\alpha} (\bfM_z^{\alpha})^{*} h, \xi \Big\rangle \\
		& \,=\, \langle h, \xi \rangle - \sum_{\alpha \in \mathbb{Z}^d_{+} \setminus \{0\}} b_{\alpha}\, \big\langle \bfM_z^{\alpha} (\bfM_z^{\alpha})^{*} h, \xi \big\rangle \\
		& \,=\, 0.
	\end{align*}
	Consequently, $\big\langle (E_{0} \otimes I_{\mathcal E}) h, \xi \big\rangle = 0$ for all $\xi \in \mathcal U^{\perp},$ and therefore
	\[
	(E_{0} \otimes I_{\mathcal E})(\mathcal{M}^\perp) \subseteq \mathcal{M}^\perp.
	\]
\end{proof}

\begin{proposition} \label{lem:smallest reducing subspace}
	Let $\mathcal{R}$ be the smallest reducing subspace for $\bfM_{\bm z}$ containing $\mathcal{M}^\perp$. Then 
	\[
	\mathcal{R}
	\,=\,
	\overline{\operatorname{span}}\left\{\bfM_{\bm z}^\alpha(\mathcal{M}^\perp) \ :\  \alpha\in\mathbb{Z}_{+}^d\right\}.
	\]
\end{proposition}
\begin{proof}
	Let us denote $
	\overline{\operatorname{span}}
	\left\{
	\bfM_{\bm z}^\alpha(\mathcal M^\perp)
	\ :\ \alpha\in\mathbb Z_+^d
	\right\}$ by $\mathcal U$. Taking $\alpha = 0$, we see that $\mathcal{U}$ contains $\mathcal{M}^\perp$. It is clear that any reducing subspace containing $\mathcal M^{\perp}$ will contain $\mathcal U.$ We only need to that $\mathcal{U}$ is a reducing subspace. Clearly, $\mathcal{U}$ is invariant under $M_{z_{i}}$ for each $i = 1, \dots d.$ Thus, it remains to show that $\mathcal{U}$ is invariant under $M_{z_{i}}^*$ for each $i.$  
	Let $\beta, \gamma\in\mathbb{Z}_{+}^d$. By an application of Lemma~\ref{lemma_compact} and then Lemma \ref{lemma_CI}, we obtain 
	\[
	\left( (\bfM_{\bm z}^\gamma)^{*} \bfM_{\bm z}^\beta - \bfM_{\bm z}^\beta (\bfM_{\bm z}^\gamma)^{*} \right)(\mathcal M^\perp)
	\, \subseteq  \, \mathcal U.
	\]
	Thus the space $\mathcal U$ is invariant under $M_{z_{i}}^{*}$ for each $i =1, \dots, d.$ 
\end{proof}

We also need the following elementary lemma concerning reducing subspaces of a unitarily invariant space. 

\begin{lemma} \label{lem:reducing subspace}
	Let \(k\) be a unitarily invariant kernel on \(\mathbb B_d\) such that $M_{z_{i}}$ is bounded on $H_k$ for all $i=1, \dots,d.$ Let
	\(\mathcal E\) be a Hilbert space. If \(\mathcal M\) is a reducing subspace for $\bfM_{\bm z} \otimes I$ on \( H_k\otimes \mathcal E\), then there exists a closed subspace
	\(\ell\subseteq \mathcal E\) such that
	\[
	\mathcal M \,=\,  H_k\otimes \ell .
	\]
	Conversely, every subspace of the form \( H_k\otimes \ell\), with
	\(\ell\subseteq \mathcal E\) closed, is reducing for
	\(\bfM_{\bm z}\) on $H_k \otimes \mathcal E$.
\end{lemma}

\begin{proof}
	First observe that
	\[
	\bigcap_{i=1}^d \ker M_{z_{i}}^*
	\,=\,
	1\otimes\mathcal E.
	\]
	Let \(\mathcal M\subseteq  H_k\otimes\mathcal E\) be reducing for \(\bfM_{\bm z}\), and let \(P_{\mathcal M}\) be the orthogonal projection of $H_k \otimes \mathcal E$ onto $\mathcal M$. Since \(\mathcal M\) reduces \(\bfM_{\bm z}\), we have
	\[
	P_{\mathcal M} M_{z_{i}} 
	\,=\,
	M_{z_{i}} P_{\mathcal M}, \qquad P_{\mathcal M} M_{z_{i}}^{*} \,=\, M_{z_{i}}^{*} P_{\mathcal M}.
	\]
	for all $i=1, \dots, d.$ For \(e\in\mathcal E\),
	\[
	M_{z_{i}}^{*} P_{\mathcal M} (1\otimes e) \,=\, P_{\mathcal M} M_{z_{i}}^{*} (1\otimes e) \,=\, 0.
	\]
	Thus
	\[
	P_{\mathcal M} (1\otimes e) \,\in\, \bigcap_i\ker M_{z_{i}}^{*} \,=\, 1\otimes\mathcal E.
	\]
	Hence there is an operator \(P\in B(\mathcal E)\) such that
	\[
	P_{\mathcal M} (1\otimes e) \,=\, 1\otimes Pe .
	\]
	Since \(P_{\mathcal M}\) is an orthogonal projection, \(P\) is an orthogonal projection. Put $ \ell=\operatorname{ran}P. $
	Then \(\ell\) is closed. Now, for every multi-index \(\alpha\),
	\[
	P_{\mathcal M} (z^\alpha\otimes e)
	\,=\,
	P_{\mathcal M} \bfM_{\bm z}^\alpha(1\otimes e)
	\,=\,
	\bfM_{\bm z}^\alpha P_{\mathcal M} (1\otimes e)
	\,=\,
	\bfM_{\bm z}^\alpha(1\otimes Pe)
	\,=\,
	z^\alpha\otimes Pe .
	\]
	By density of vector-valued polynomials,
	\[
	P_{\mathcal M} \,=\, I_{H_k}\otimes P.
	\]
	Therefore
	\[
	\mathcal M \,=\, \operatorname{ran}Q
	\,=\, 
	 H_k\otimes \operatorname{ran}P
	\,=\,
	 H_k\otimes \ell .
	\]
	Conversely, every subspace of the form \(H_k\otimes\ell\), with
	\(\ell\subseteq\mathcal E\) closed, is clearly reducing for \(\bfM_{\bm z}\).
	Hence the reducing subspaces are exactly
	\[
	 H_k\otimes \ell ,
	\qquad \ell\subseteq\mathcal E \text{ closed}.
	\]
\end{proof}

\section{Proofs of Main Results}\label{section_proof}
As the title suggests, we prove all the main results in this section. We first require the following lower estimate for the Taylor joint spectrum. The proof of this theorem is motivated by \cite{DEHS}. 
\begin{theorem}\label{theorem_tylor_lower}
	Let $k$ be a unitarily invariant CNP kernel on $\mathbb B_d$. Let ${\bfT}=(T_1, \cdots, T_d)$ be a $1/k$-contraction such that $ \dim({\rm Ran} \Delta_{\bfT}) < \infty.$ Let $
	\theta_{{\bfT}} : \mathbb{B}_d \to \mathcal{B}(\mathcal{D}_{{\tilde{\bfT}}}, {\Delta}_{\bfT}) $
	be the characteristic function of ${\bfT}$. Then
	\[
	\sigma(\bfT) \supset E_{\theta_{\bfT}}\cup F_{\theta_{\bfT}}.
	\]
	
\end{theorem}
\begin{proof}
	
	It is easy to see that $F_{\theta_{\bfT}} \subset \sigma(\bfT)$. Indeed, if there is a point 
	$\bm{\lambda} \in \overline{\mathbb{B}}_d$ which is not in $\sigma(\bfT)$, then, by 
	Proposition \ref{prop:continuous-extension-characteristic-function}, the characteristic function of $\bfT$ extends continuously at $\bm\lambda$, and hence $\bm{\lambda} \notin F_{\theta_{\bfT}}$. Thus, our goal is now to show that $E_{\theta_{\bfT}} \subset \sigma(\bfT)$. 
	
	\smallskip
	
	For this, we shall show that for each point $\bm{z} \in \overline{\mathbb{B}}_d$ which is not in the Taylor joint spectrum of $\bfT$, the characteristic function $\widehat\theta_{\bfT}(\bm z)$ is surjective. We know by Proposition  \ref{prop:continuous-extension-characteristic-function} that the characteristic function of $\bfT$ extends continuously to $
	\Omega := \overline{\mathbb{B}}_{d} \setminus \sigma(\bfT) 
	\,\subseteq\,
	\overline{\mathbb{B}}_{d}.$
	It is enough to show that $\widehat{\theta}_{\bfT}(\bm{z})^*$ is injective, since $\operatorname{Ran}\Delta_{\bfT}$ is finite-dimensional. Recall that
	\[
	\widehat{\theta}_{{\bfT}}(\bm z)^*
	\,=\,
	- {\tilde{\bfT}}^* + D_{{\tilde{\bfT}}} \, {\bfZ}^*(I_{\mathcal{H}} - {\tilde{\bfT}} \bfZ^*)^{-1}\Delta_{{\bfT}}
	\]
	for all $\bm z \in \Omega$. A direct calculation yields the following 
	\[
	\mathcal{D}_{\tilde{\bfT}} \, \widehat{\theta}_{{\bfT}}(\bm z)^*
	\,=\,
	({\bfZ}^* - {\tilde{\bfT}}^*)(I_{\mathcal{H}} - {\tilde{\bfT}} \bfZ^*)^{-1}\Delta_{{\bfT}}.
	\]
	If $y \in \operatorname{Ran}\Delta_{\bfT}$, then there exists
	$x \in \cH$ such that $y = \Delta_{\bfT} \, x$.  Fix $\bm z\in\Omega$, assume that $\widehat{\theta}_{{\bfT}}(\bm z)^* y = 0$. Then
	$ \mathcal{D}_{\tilde{\bfT}}\,\widehat{\theta}_{\bfT}(\bm z)^* y = 0.$ Using the above equation, this is equivalent to
	$
	(\bfZ^* - \tilde{\bfT}^*)(I_{\mathcal{H}} - \tilde{\bfT}\bfZ^*)^{-1}
	\bigl(I_{\mathcal{H}} - \tilde{\boldsymbol{T}}\,\tilde{\bfT}^{*}\bigr)x = 0$.

    \smallskip
    
\noindent\textbf{Claim} $(\bfZ^* - \tilde{\bfT}^*)$ is injective on $\overline{\mathbb B}_d\setminus\sigma(\bfT).$ 

\smallskip
    
    To prove the claim, let $h\in\mathcal{H}$ be such that $({\bfZ}^* - \tilde{{\bfT}}^*)h = 0$. Then
	$
	\psi_{\alpha}({\bfT})^* h = \overline{\psi_{\alpha}(\bm z)} \, h,$ for all $\alpha\in\mathbb{Z}_+^d$. 
	By the definition of $\psi_{\alpha}$, we get 
	\[
	\sqrt{b_{\alpha}}\, ({{\bfT}^{\alpha}})^* h = \sqrt{b_{\alpha}}\, \overline{{\bm z}^{\alpha}}\, h \quad \text{ for all } \alpha\in\mathbb{Z}_+^d.
	\]
	Since $b_{\alpha} \neq 0$ whenever $|\alpha| = 1$, we obtain
	\[
	T_i^* h = \overline{z}_i\, h, \qquad i = 1,\dots,d.
	\]
	Hence,
	$
	(\overline{z}_1,\dots,\overline{z}_d) \in \sigma_p(T_1^*,\dots,T_d^*).$
	Using \eqref{equation_Tylor1} and \eqref{equation_Tylor2}, it follows that
	\[
	(z_1,\dots,z_d) 
	\,\in\, \sigma(T_1,\dots,T_d),
	\]
	which is a contradiction since $\bm z$ lies outside the spectrum. This completes the proof of the claim.

   \noindent By the claim, the equality $
(\mathbf Z^*-\tilde{\mathbf T}^*)
(I_{\mathcal H}-\tilde{\mathbf T}\mathbf Z^*)^{-1}
\bigl(I_{\mathcal H}-\tilde{\mathbf T}\tilde{\mathbf T}^*\bigr)x=0$
forces $
\bigl(I_{\mathcal H}-\tilde{\mathbf T}\tilde{\mathbf T}^*\bigr)x=0.$ Hence,
	\[
	\|y\|^2 
	\,=\,
	\langle \Delta_{\bfT}x, \Delta_{\bfT}x \rangle 
	\,=\, 
	\langle \Delta_{\bfT}^2 x, x \rangle
	\,=\,
	\big\langle \bigl(I_{\mathcal{H}} - \tilde{\bfT}\,\tilde{\bfT}^{*}\bigr)x, x \big\rangle
	\,=\, 0.
	\]
	Therefore, $y = 0$, and hence $\widehat{\theta}_{\bfT}(\bm z)^*$ is injective. Thus, the proof is complete.  
	
\end{proof}

\underline{\bf Proof of Theorem A:} Part~(i) is proved in Proposition \ref{prop:continuous-extension-characteristic-function}, part~(ii) is proved in \cref{theorem_tylor_lower}, and part~(iii) follows from Lemma \ref{lemma_upper}.

\quad 

To prove our second result, we first establish the following result. 
\begin{theorem} \label{thm:BLH}
	Let $k$ be a regular unitary invariant CNP kernel on $\mathbb{B}_{d}$. Let $\mathcal E$ be a Hilbert space and $\mathcal M \subseteq H_k \otimes \mathcal E$ be a closed subspace invariant under $M_{z_{i}} \otimes I_{\mathcal E}$ on $H_k \otimes \mathcal{E}$  for all $i=1, \dots, d.$ Then 
	\[
	\mathcal M \,=\, (H_k \otimes \cE_{1}) \oplus (I_{H_k} \otimes \tau) \theta_{\bfT}  (H_k \otimes \mathcal D_{\tilde{\bfT}}),
	\]
	where $\cE_{1}$ is a closed subspace of $\cE$ such that $H_k \otimes \cE_{1}$ is the largest reducing subspace contained in $\cM,$  $\theta_{\bfT}$ is the characteristic function of the pure $1/k$-contraction 
	\[
	\bfT \,=\, \Big(P_{\mathcal M^{\perp}} (M_{z_{1}} \otimes I_{\mathcal E})|_{\mathcal M^{\perp}}, \dots, P_{\mathcal M^{\perp}} (M_{z_{d}} \otimes I_{\mathcal E})|_{\mathcal M^{\perp}} \Big),
	\]
	and $\tau: \overline{{\rm Ran}} \Delta_{\bfT} \to \mathcal E \ominus \cE_{1}$ is a unitary. 
\end{theorem}
\begin{proof}
Let $\mathcal R = \overline{\rm span} \big\{ \bfM_{\bm z}^{\alpha} (\mathcal M^{\perp}) \ :\ \alpha \in \mathbb{Z}_{+}^{d} \big\}.$ By Proposition \ref{lem:smallest reducing subspace}, $\mathcal R$ is the smallest reducing subspace containing $\mathcal M^{\perp}.$ Then by Lemma \ref{lem:reducing subspace}, we get $\mathcal R = H_k \otimes (\mathcal R \cap \cE)$. 

    \smallskip
	
	Since $\bfT = P_{\mathcal M^{\perp}} (\bfM_{\bm z} \otimes I_{\cE})|_{\mathcal M^{\perp}}$ is a pure $1/k$-contraction, $V_{\bfT}$ is an isometry, see \cite[Theorem 2.2]{BJ} and \cite[Theorem 1.3]{AE}.
	Moreover, this dilation is
	minimal. Indeed, any reducing subspace of $H_k \otimes I_{\overline{\rm Ran} \Delta_{\bfT}}$ which contains \(V_T\mathcal H\) must contain the constant coefficients
	\(\Delta_T h\) \  \( (h\in\mathcal H)\), and hence must be all of
	\(H_k\otimes\overline{\rm Ran} \Delta_{\bfT}\). Therefore, we can write 
	\[
	H_k\otimes\overline{\rm Ran} \Delta_{\bfT} 
	\,=\, 
	\overline{\rm span} \big\{ (\bfM_{\bm z}^{\alpha} \otimes I_{\overline{\rm Ran} \Delta_{\bfT}}) V_{\bfT} h\ :\ h\in \mathcal M^{\perp}, \, \alpha \in \mathbb{Z}^{d}_{+} \big\}. 
	\]
	Since the inclusion map $i: M^\perp \hookrightarrow H_k\otimes (\mathcal R \cap \mathcal E)$ and the canonical dilation $V_{\bfT}: \mathcal M^{\perp} \to H_k \otimes \overline{\rm Ran} \Delta_{\bfT}$ are both minimal diltions for $\bfT$, there is a unique unitary operator $ U: H_k \otimes \overline{\rm Ran}\Delta_{\bfT} \to H_k \otimes (\mathcal R \cap \mathcal E)$ such that 
	\[
	U V_{\bfT} = i, \quad \text { and } \quad U (M_{z_{j}} \otimes I_{\overline{\rm Ran} \Delta_{\bfT}}) 
	\,=\,
	(M_{z_{j}} \otimes I_{( \mathcal R \cap \cE)} ) U,  
	\]
	for $j=1,\cdots, d$. Since $U$ intertwines $\bfM_{\bm z} \otimes I$ on suitable spaces, $U = M_{\varphi}$ for some $\varphi \in {\rm Mult}(H_k \otimes \overline{\rm Ran}\Delta_{\bfT}, H_k \otimes (\mathcal{R} \cap \cE)).$ There are no non-constant unitary multipliers other than constant unitaries. Therefore, $U = I_{H_k} \otimes \tau$ for some unitary $\tau: \overline{{\rm Ran}} \Delta_{\bfT} \to (\mathcal R \cap \mathcal E).$ Now note that
	\begin{align*}
	(I_{H_k} \otimes \tau) M_{\theta_{\bfT}}  (H_k \otimes \mathcal D_{\tilde{\bfT}})
	& \,=\,
	(I_{H_k} \otimes \tau) ((\rm Ran V_{\bfT})^{\perp}) \\
	& \,=\, 
	\big( ( I_{H_k} \otimes \tau) (\rm Ran V_{\bfT}) \big)^{\perp} \\
	& \,=\,
	(H_k \otimes (\mathcal R \cap \cE)) \ominus \mathcal M^{\perp} \\
    & \,=\, \cM \ominus (H_k \otimes \cE_{1}),
	\end{align*}
	where $\cE_{1} = \mathcal E \ominus (\mathcal R \cap \mathcal E).$ Finally, since $H_k \otimes (\mathcal R \cap \cE)$ is the smallest reducing subspace containing $\cM^{\perp},$ $H_k \otimes \cE_{1}$ is the largest reducing subspace contained in $\cM.$ 
\end{proof}
The Drury--Arveson version of \cref{thm:BLH} was established in \cite{BES2}. As an application of \cref{thm:BLH}, we obtain that  every inner multiplier can be realized as the characteristic function of a pure $1/k$-contraction.

\begin{theorem}\label{theorem_main2}
Let $k$ be a regular unitarily invariant CNP kernel on $\mathbb{B}_{d}.$ Let $\cE$ and $ \cF$ be two Hilbert spaces and $\theta \in {\rm Mult}(H_k \otimes \cF, H_k \otimes \cE)$ be an inner multiplier. Then  
\[
\mathcal M 
\,:=\,
\theta (H_k \otimes \cF) \,=\, (H_k \otimes \cE_{1}) \oplus 
(I_{H_k} \otimes \tau) \theta_{\bfT} (H_k \otimes \mathcal{D}_{\tilde{\bfT}}), 
\]
where $\cE_{1}$ is a closed subspace of $\cE$ such that $H_k \otimes \cE_{1}$ is the largest reducing subspace contained in $\theta (H_k \otimes \cF)$, $\theta_{\bfT}$ is the characteristic function of the pure $1/k$-contraction 
\[
\bfT \,=\, \Big(P_{\mathcal M^{\perp}} (M_{z_{1}} \otimes I_{\mathcal E})|_{\mathcal M^{\perp}}, \dots, P_{\mathcal M^{\perp}} (M_{z_{d}} \otimes I_{\mathcal E})|_{\mathcal M^{\perp}} \Big),
\]
and $\tau: \overline{{\rm Ran}} \Delta_{\bfT} \to \mathcal E \ominus \mathcal E_{1}$ is a unitary. Therefore, there exists a partial isometry $V: \mathcal E_{1} \oplus \mathcal D_{\bfT} \to \mathcal F$ such that 
\[
I_{\mathcal E_{1}} \oplus \tau\, \theta_{\bfT} (\bm z)
 \ =\ 
\theta (\bm z) V \qquad \text{and} \qquad \theta(\bm z) 
\ =\  ( I_{\mathcal E_{1}} \oplus \tau\, \theta_{\bfT} (\bm z) ) V^{*} \quad \text{ for all } \bm z \in \mathbb{B}_{d}.
\]
\end{theorem}
\begin{proof}
    Apply \cref{thm:BLH} for the closed invariant subspace $\cM = \theta (H_k \otimes \cF).$
\end{proof}

Now, we are ready to proof our second main result. 

\smallskip

\underline{\bf Proof of Theorem B:} This follows immediately as a consequence of \cref{theorem_tylor_lower} and \cref{theorem_main2}.

\section{Epilogue}

\cref{thmA,thmB} provide lower estimates for the Taylor joint spectrum of pure $1/k$-contractions and quotient modules, respectively. An upper estimate was obtained by Clou\^atre--Timko in \cite{CT}. In the case of the Drury--Arveson kernel, the equality between $\sigma(\bfT)$ and $E_{\theta_{\bfT}}\cup F_{\theta_{\bfT}}$ was established in \cite{DEHS}. Their proof relies on \cite[Lemma 5]{DEHS}, together with the corona property. We observe that the same lemma admits an analogue for CNP kernels, with the necessary modifications and essentially the same argument. For completeness, we record this version below.

\begin{lemma}\label{lem:determinant-ideal-CNP}
	Let $k$ be a unitarily invariant CNP kernel on $\mathbb B_d$. Let $\mathcal D$ be finite-dimensional Hilbert space with orthonormal basis $d_1,\ldots,d_N$ and $ \mathcal M\in \operatorname{Lat}(\bfM_{\bm z},H_k \otimes \mathcal D)$ an invariant subspace. 
	Let $\cE$ be a Hilbert space and $\theta\in \operatorname{Mult}(H_k \otimes \mathcal E,H_k \otimes \mathcal D)$ satisfy $ M_\theta (H_k \otimes \mathcal E) \subseteq \mathcal M.$
	Fix vectors $e_1,\ldots,e_N\in\mathcal E$ and denote by $\Theta=(\theta_{ij})_{i,j=1}^N \in M_{N} (\rm Mult (H_k))$ the matrix whose coefficients are determined by
	\[
	\theta(\bm z)e_j
	\,=\,
	\sum_{i=1}^N \theta_{ij}(\bm z)d_i, \qquad (j=1, \dots, N,\ \bm z \in \mathbb{B}_{d}) .
	\]
	Then $ \det\Theta\in I(\mathcal M),$
	where $
	I(\mathcal M):=\{f\in \operatorname{Mult}(H_k): f (H_k \otimes \mathcal D)\subseteq \mathcal M\}. $ If $\bm \lambda \in \overline{\mathbb{B}}_{d}$ is a point such that  $\widehat{\theta} (\bm \lambda)
	=
	\lim_{\substack{\bm z \to \lambda \\ \bm z \in \mathbb{B}_{d}}} \theta (\bm z) $
	exists and $\widehat{\theta}(\bm \lambda)$ is onto, then there exists $f\in I(\mathcal M)$ such that $\lim_{\substack{\bm z \to \lambda \\ \bm z \in \mathbb{B}_{d}}} f(\bm z)=1.$
\end{lemma}

\begin{proof}
	Since the entries of \(\Theta\) are scalar multipliers, it follows that $ g:=\det \Theta $ is also a scalar multiplier. We claim that \(g\in I(\mathcal M)\). Let \(\operatorname{adj}(\Theta)\) denote the adjugate matrix of
	\(\Theta\). Then
	\[
	\Theta\,\operatorname{adj}(\Theta)
	\,=\,
	(\det\Theta) \cdot I_N
	=
	g \cdot I_N .
	\]
	Let \(\varphi_{jk}\) denote the \((j,k)\)-entry of \(\operatorname{adj}(\Theta)\). Each \(\varphi_{jk}\) is again a scalar multiplier. For every $1 \leq i,k \leq N,$
	\[
	\sum\limits_{j=1}^{N} \theta_{i,j} \cdot \varphi_{j,k} 
	\,=\, 
	g \, \delta_{i,k}. 
	\]
	Therefore, for every \(h\in H_k\) and every \(k=1,\ldots,N\), we have
	\begin{align*}
		\sum\limits_{j=1}^{N} (h \cdot \varphi_{j,k}) ( 1 \otimes \theta e_{j} ) 
		& \,=\,
		\sum\limits_{j=1}^{N} (h \cdot \varphi_{j,k})   \left( \sum\limits_{i=1}^{N} \theta_{i,j} d_{i} \right)
		\,=\,
		\sum\limits_{i=1}^{N} h \left( \sum\limits_{j=1}^{N} \varphi_{j,k} \cdot \theta_{i,j}  \right) d_{i}
		\,=\,
		\sum\limits_{i=1}^{N} h (g\, \delta_{i,k}) \,d_{i} \\
		& \,=\, (g \cdot h)\otimes d_k. 
	\end{align*}
	Now \(h  \cdot \varphi_{jk}\in H_k\), and hence $ (h \cdot \varphi_{jk}) \otimes e_j\in H_k \otimes \mathcal E.$
	Since \(M_\theta ( H_k \otimes \mathcal E )\subseteq \mathcal M\), we get
	\[
	M_\theta((h \cdot  \varphi_{jk}) \otimes e_j)
	\,=\,
	(h \cdot \varphi_{jk})\,\theta ( 1 \otimes e_j) 
	\,\in\, \mathcal M .
	\]
	Thus every term in the sum belongs to \(\mathcal M\), and hence $ (g \cdot h) \otimes d_k\in \mathcal M .$
	Since \(h\in H_k\) and \(k=1,\ldots,N\) were arbitrary, it follows that
	\[
	g ( H_k \otimes \mathcal D) \,\subseteq\, \mathcal M .
	\]
	Therefore $ g=\det\Theta\in I(\mathcal M).$
	
	\smallskip
	
	For the remaining part of the assertion fix $\bm \lambda \in \overline{\mathbb{B}}_{d}$
	such that
	\[
	\tilde\theta(\bm\lambda)
	=
	\lim_{\substack{\bm z\to \bm\lambda\\ \bm z\in\mathbb B_d}}
	\theta(\bm z)
	\]
	exists and is onto. Since \(\dim \mathcal D=N\), we may choose
	\(e_1,\ldots,e_N\in\mathcal E\) such that
	\[
	\tilde\theta(\bm\lambda)e_j \,=\, d_{j}, \qquad (j=1, \dots N).
	\]
	Let $\Theta = (\theta_{i,j})_{i,j}$ be the matrix formed as above with respect to the vectors $e_{1}, \dots, e_{N}$ chosen in this way. Then the limit $ \lim_{\substack{\bm z\to \bm\lambda\\ \bm z\in\mathbb B_d}}
	\Theta(\bm z)$ exists and is equal to $I_{N}.$ Hence $
	f \,=\, \det(\Theta)$ defines a function in $ I(\mathcal M)$ as in the statement of the Lemma.  
\end{proof}

\begin{lemma}\label{lemma_upper}
Let $k$ be a unitarily invariant CNP kernel on $\mathbb{B}_d$ which has the corona property, let $\mathcal{D}$ be finite-dimensional, and let 
$\mathcal{M} \in \operatorname{Lat}(M_z,H^2_s(\mathcal{D}))$ 
be closed. Suppose that 
$\theta : \mathbb{B}_d \to \mathcal{L}(\mathcal{E},\mathcal{D})$ 
is an inner multiplier such that $
\theta H^2_s(\mathcal{E}) \subseteq \mathcal{M}.$
Then
\[
\sigma(\bfM_z,H^2_s(\mathcal{D})/\mathcal{M})
\,\subseteq\,
E_{\theta}\cup F_{\theta}.
\]
\end{lemma}
\begin{proof}
Let $I(\mathcal M)$ be the ideal appearing in Lemma \ref{lem:determinant-ideal-CNP}. For $f\in I(\mathcal M)$, set
\[
AZ(f)
\,=\,
\left\{
\lambda \in \overline{\mathbb{B}}_d \  :\ 
\liminf_{z \to \lambda} |f(z)| = 0
\right\}.
\]
Then by \cite[Corollary 3.14]{CT}, we have
\[
\sigma(\bfM_z,H^2_s(\mathcal D)/\mathcal M)
\,\subseteq \,
\bigcap_{f\in I(\mathcal M)} AZ(f).
\]
Next, the inclusion 
\[
\bigcap_{f\in I(\mathcal M)} AZ(f) 
\, \subseteq\, 
E_{\theta} \cup F_{\theta}
\]
follows from Lemma~\ref{lem:determinant-ideal-CNP}.  Indeed, suppose $\lambda \in \overline{\mathbb{B}}_d$ is such that
$ \lim_{\substack{z \to \lambda \\ z \in \mathbb{B}_d}} \theta(z) $
exists and is surjective. Then, by Lemma~\ref{lem:determinant-ideal-CNP}, there exists a function $f \in I(\mathcal{M})$ such that
	$\lim_{\substack{z \to \lambda \\ z \in \mathbb{B}_d}} f(z) = 1.$
	Hence $\lambda \notin \bigcap_{f\in I(\cM)}AZ(f)$.
\end{proof}

\end{document}